\documentclass[a4paper]{amsart}
\usepackage[utf8x]{inputenc}
\usepackage[all]{xy}
\usepackage{amsmath,amsthm,amssymb,latexsym,epic,bbm,comment,mathbbol, mathrsfs}

\usepackage{graphicx,enumerate,stmaryrd,color}

\usepackage{IEEEtrantools}

\usepackage{stmaryrd}

\usepackage{hhline}

\usepackage{tikz}
\usetikzlibrary{matrix,arrows,decorations.pathmorphing}
\usetikzlibrary{positioning}
\usepackage[parfill]{parskip}
\usepackage[active]{srcltx}

\title[$2$-representations of $2$-subcategories of projective functors]{Simple transitive $2$-representations of 
left cell $2$-subcategories of projective functors for star algebras}
\author{Jakob Zimmermann}
\date{\today}

\newtheorem{theorem}{Theorem}[section]
\newtheorem{proposition}[theorem]{Proposition}
\newtheorem{lemma}[theorem]{Lemma}

\theoremstyle{definition}

\newtheorem{remark}[theorem]{Remark}

\newtheorem{conjecture}{Conjecture}

\newcommand{\setof}[2]{\{#1\;|\; #2\}}
\newcommand{\on}[1]{\operatorname{#1}}

\font\sc=rsfs10

\newcommand{\cC}{\sc\mbox{C}\hspace{1.0pt}}

\renewcommand{\l}{\mathcal{L}}

\usepackage{enumerate}

\begin{document}

\begin{abstract}
In this paper we study simple transitive $2$-representations of certain $2$-subcategories of the $2$-category of 
projective functors over a star algebra. We show that in the simplest case, which is associated to the Dynkin type $A_2$, simple 
transitive $2$-representations are classified by cell $2$-representations. In the general case we conjecture that there
exist many simple transitive $2$-representations which are not cell $2$-representations and provide some evidence for our conjecture.
\end{abstract}

\maketitle
\section{Introduction}\label{s1}
In 2010 Mazorchuk and Miemietz (\cite{MM1, MM2, MM3, MM4, MM5, MM6}) started an abstract study of $2$-representations 
of finitary $2$-categories. This was inspired by the categorification philosophy and the related use of categorical actions 
which appeared, among others, in the celebrated categorification of the Jones polynomial; see \cite{Kh}, and in the proof of 
Broué's abelian defect group conjecture for symmetric groups in \cite{CR}. For more details about classification problems in 
$2$-representation theory we refer the reader to the survey \cite{Maz2} and references therein.

In representation theory of finite dimensional (associative) algebras simple modules play an important role as
every finite dimensional module can be constructed from simple modules by repeated extensions, which leads to the classical
Jordan-H\"{o}lder theory. A weak analogue of Jordan-H{\"o}lder theory for $2$-representations of finitary $2$-categories
was developed in \cite{MM5}, where the role of simple modules is played by the so-called simple transitive $2$-representations. 
Since then a lot of research has been devoted 
to try to classify simple transitive $2$-representations for various finitary $2$-categories fitting the framework of \cite{MM5}, see
e.g.,\ \cite{MM5, KMMZ, Zi1, Zi2, MZ1, MMZ, MMZ}.

One of the easiest examples of a finitary $2$-category is the $2$-category of projectives functors $\cC_{\Lambda}$ over a given finite
dimensional associative algebra $\Lambda$. This $2$-category consists of one object which can be thought of as the category 
$\Lambda$-$\on{proj}$ of finite dimensional projective (left) $\Lambda$-modules, $1$-morphisms are, morally, endofunctors of 
$\Lambda$-$\on{proj}$ isomorphic to tensoring with direct sums of projective and regular $\Lambda$-$\Lambda$-bimodules and 
$2$-morphisms are natural transformations. A classification of simple transitive $2$-representations of $\cC_{\Lambda}$, for $\Lambda$
self-injective, is given already in \cite[Theorem 15]{MM5}. It turns out that every simple transitive $2$-representations is
equivalent to a so-called cell $2$-representation. The latter $2$-representations were defined in \cite{MM1} and associated 
naturally to the combinatorial structure of the $2$-category which is usually referred to as the \emph{cell structure}.
Recently, the fact that simple transitive $2$-representations of $\cC_{\Lambda}$ are exhausted by cell $2$-representations was proved 
in full generality (that is for arbitrary $\Lambda$) in \cite{MMZ}. Some intermediate results appeared in \cite{MZ1, Zi2, MMZ0}.

The $2$-category $\cC_{\Lambda}$ has many left and right cells. Roughly speaking, they are indexed by isomorphism classes of 
indecomposable projective $\Lambda$-modules. In the present paper we look at a certain $2$-subcategory of $\cC_{\Lambda}$ in which 
we keep only one of these left cells without any restrictions on the right cells. The question we ask is: what are simple 
transitive $2$-representations of this $2$-subcategory. In this way we hope to break the ``left-right symmetry'' of $\cC_{\Lambda}$
and produce a $2$-category with a completely different behaviour. The disadvantage of this situation is that 
none of the general approaches of \cite{MM5, MMZ} is applicable. Therefore, we are able to answer the question 
only for some very special algebra $\Lambda_1$. For this algebra the answer is similar to that for $\cC_{\Lambda}$. However,
the algebra $\Lambda_1$ is naturally a representative of a family of algebras denoted $\Lambda_n$ which
are associated to certain star shaped quivers. We also investigate simple transitive $2$-representations
for similar $2$-subcategories for all algebras $\Lambda_n$. The behaviour which we observe suggests that, for 
$n > 1$, simple transitive $2$-representations are not exhausted by cell $2$-representations. We provide some evidence 
to substantiate this claim. However, at the present stage we are not able to give an explicit construction of non cell 
simple transitive $2$-representations, but we indicate how they should look like.

The paper is organized in the following way. In the next section we introduce all necessary notions and notation. 
In Section \ref{s3} we generalize \cite[Theorem 11]{KMMZ}, under some additional assumptions, which ensures that even in the setup
of the present paper the action of $1$-morphisms on a simple transitive $2$-representation maps any object to a projective object.
In Section \ref{s4} we classify simple transitive $2$-representations of our $2$-category in the case of the algebra $\Lambda_1$.
In Section \ref{s5} we describe numerical invariants of simple transitive $2$-representations in the general case of $\Lambda_n$, 
where $n > 1$. Based on these invariants, we formulate a conjecture about classification of simple transitive $2$-representations. 
In the final section we look at $2$-subcategories of $\cC_{\Lambda_n}$ where we keep only one right cell and 
in this case we show that simple transitive $2$-representations are classified by cell $2$-representations.

\section{Preliminaries}\label{s2}
Let $\Bbbk$ be an algebraically closed field. We write $\otimes$ for $\otimes_{\Bbbk}$.

\subsection{Finitary $2$-categories and their representations}\label{s2.1}
A \emph{$2$-category} is a category enriched over the category \textbf{Cat} of all small categories. Hence a $2$-category $\mathscr{C}$ 
consists of the following data:
\begin{itemize}
 \item objects, denoted by $\mathtt{i}, \mathtt{j}, \mathtt{k}, \ldots$;
 \item For each $\mathtt{i},\mathtt{j}$, a small category $\mathscr{C}(\mathtt{i},\mathtt{j})$ of $1$-morphisms 
 which we denote by $F, G, H,  \ldots$ and bifunctorial composition, denoted by $\circ$;
 \item indecomposable identity $1$-morphisms $\mathbb{1}_{\mathtt{i}}$;
 \item $2$-morphisms denoted by $\alpha, \beta, \gamma, \ldots$;
 \item identity $2$-morphisms $\text{id}_{F}$;
 \item and horizontal and vertical composition of $2$-morphisms denoted by $\circ_h$ and $\circ_v$, respectively.
\end{itemize}
Moreover, all the obvious strict axioms are satisfied. We often write $F(\alpha)$ for $\text{id}_F \circ_h \alpha$ and
$\alpha_F$ for $\alpha \circ_h \text{id}_F$.

Recall that a category $\mathcal{C}$ is called \emph{finitary $\Bbbk$-linear}, if it is equivalent to $B$-proj for some $\Bbbk$-linear, 
associative, unital, finite dimensional, basic algebra $B$. We say that a $2$-category $\mathscr{C}$ is \emph{finitary} if it has 
finitely many objects, each morphism category $\mathscr{C}(\mathtt{i},\mathtt{j})$ is finitary $\Bbbk$-linear and all compositions are 
biadditive and $\Bbbk$-(bi)linear, whenever applicable.

Examples of $2$-categories are:
\begin{itemize}
 \item The $2$-category $\textbf{Cat}$ of small categories whose objects are small categories, $1$-morphisms are functors and 
 $2$-morphisms are natural transformations.
 \item The $2$-category $\mathfrak{A}_{\Bbbk}^f$ of \emph{finitary $\Bbbk$-linear} categories, whose objects are finitary $\Bbbk$-linear
 categories, $1$-morphisms are additive $\Bbbk$-linear functors and $2$-morphisms are natural transformations.
 \item The $2$-category $\mathfrak{R}_{\Bbbk}$ of \emph{finitary $\Bbbk$-linear abelian} categories, whose objects are categories 
 equivalent to module categories of finite dimensional associative $\Bbbk$-algebras, $1$-morphisms are right exact additive $\Bbbk$-linear 
 functors and $2$-morphisms are natural transformations.
\end{itemize}

From now on let $\mathscr{C}$ denote a finitary $2$-category. 

A \emph{finitary $2$-representation} of $\mathscr{C}$ is a (strict) $2$-functor $\mathbf{M}: \mathscr{C} \to \mathfrak{A}_{\Bbbk}^{f}$. 
All finitary $2$-representations form a $2$-category, denoted by $\mathscr{C}$-afmod, where $1$-morphisms are given by strong natural 
transformations and $2$-morphisms are modifications, see \cite[Subsection 2.3]{MM3}. 

Analogously we define the $2$-category $\mathscr{C}$-mod as the $2$-category of abelian $2$-re\-pre\-sen\-ta\-tions which are 
$2$-functors from $\mathscr{C}$ to $\mathfrak{R}_{\Bbbk}$. There is a way to \emph{abelianize} a finitary $2$-representation
$\mathbf{M}$, using the diagrammatic \emph{abelianization} $2$-functor 
\begin{displaymath}
 \overline{\cdot}: \mathscr{C}\text{-afmod} \to \mathscr{C}\text{-mod},
\end{displaymath}
see \cite[Subsection 3.1]{MM1}. 
We note that $\mathbf{M}$ is equivalent to $\overline{\mathbf{M}}_{\text{proj}}$; see \cite[Theorem 11]{MM2}. 
We denote $2$-representations by bold capital roman letters $\mathbf{M}, \mathbf{N}, \ldots$. 

For each $\mathtt{i} \in \mathscr{C}$, we define the \emph{principal} $2$-representation 
$\mathbf{P}_{\mathtt{i}} := \mathscr{C}(\mathtt{i},{}_-)$, for which we have the usual Yoneda lemma, see \cite[Lemma 3]{MM3}.
%

A $2$-representation $\mathbf{M}$ of $\mathscr{C}$ is called \emph{$1$-faithful} if $\mathbf{M}(F) \neq 0$ for all 
$1$-morphisms $F \in \mathscr{C}$. If $\cC$ has only one object $\mathtt{i}$, we say that a finitary $2$-representation of 
$\cC$ has rank $r$ if the algebra $B$ satisfying $\mathbf{M}(\mathtt{i}) \simeq B$-$\text{proj}$ has exactly $r$ isomorphism 
classes of indecomposable projectives.

For more details about $2$-categories and $2$-representations we refer the reader to \cite{Le, Mac, Maz1, Maz2}.

\subsection{Simple transitive $2$-representations}\label{s2.2}
A finitary $2$-representation $\mathbf{M}$ of $\mathscr{C}$ is called \emph{transitive} if, for any $\mathtt{i}$ and $\mathtt{j}$ and 
any indecomposable objects $X \in \mathbf{M}(\mathtt{i})$ and $Y \in \mathbf{M}(\mathtt{j})$, there is a $1$-morphism $F$ of $\mathscr{C}$
such that $Y$ is isomorphic to a direct summand of $\mathbf{M}(F)X$. 

A finitary $2$-representation $\mathbf{M}$ of $\mathscr{C}$ is called \emph{simple} provided that it does not have any non-zero proper
$\mathscr{C}$-invariant ideals. Note that simplicity implies transitivity, however, we will
always speak about \emph{simple transitive} $2$-representations, for historical reasons.

\subsection{Cells and cell $2$-representations}\label{s2.3}
For indecomposable $1$-morphisms $F$ and $G$ of $\mathscr{C}$, we write $F \geq_L G$ provided that there exists a $1$-morphism $H$ in
$\mathscr{C}$, such that $F$ is isomorphic to a direct summand of $H \circ G$. This defines the \emph{left} preorder $\geq_L$, whose 
equivalence classes are \emph{left cells}. Similarly one defines the \emph{right} preorder $\geq_R$ and the \emph{two-sided}
order $\geq_J$, and the corresponding \emph{right} and \emph{two-sided} cells, respectively.

By \cite{CM}, for every simple transitive $2$-representation $\mathbf{M}$ there exists a unique maximal with respect to $\geq_J$ 
two-sided cell which is not annihilated by $\mathbf{M}$; called the \emph{apex} of $\mathbf{M}$. A two-sided cell $\mathcal{J}$ of 
$\cC$ is called \emph{idempotent}, if it contains elements $F, G$ and $H$ such that $F$ is isomorphic to a direct summand of 
$G \circ H$. Observe that the apex always is an idempotent two-sided cell.

Let $\mathcal{L}$ be a left cell in $\mathscr{C}$, then there exists a unique $\mathtt{i}_{\mathcal{L}} = \mathtt{i}$ such that 
all $F \in \mathcal{L}$ start at $\mathtt{i}$. The corresponding \emph{cell $2$-representation} $\mathbf{C}_{\mathcal{L}}$ is the 
subquotient of the principal $2$-representation $\mathbf{P}_{\mathtt{i}}$ obtained by taking the unique simple transitive quotient
of the subrepresentation of $\mathbf{P}_{\mathtt{i}}$ given by the additive closure of all $1$-morphisms $F$ such that 
$F \geq_L \mathcal{L}$. The $2$-representation $\mathbf{C}_{\mathcal{L}}$ is, by construction, simple transitive. We refer the reader 
to \cite[Subsection 6.5]{MM5} for details about simple transitive and cell $2$-representations.

\subsection{The matrix of the action of a $1$-morphism}\label{s2.4}
Let $\mathbf{M}$ be a finitary $2$-representation of $\mathscr{C}$ and $F$ a $1$-morphism in $\mathscr{C}$. 
We can already say a lot about $\mathbf{M}$ by studying what we will call the \emph{matrix $[F]$ of the action of $F$} on the 
indecomposable projectives in $\mathcal{M} := \coprod_{\mathtt{i}}\mathbf{M}(\mathtt{i})$. The rows and columns of $[F]$ are indexed
by isomorphism classes of indecomposable objects in $\mathcal{M}$ and the $X \times Y$-entry encodes the multiplicity of $X$ as a direct 
summand of $\mathbf{M}(F)\,Y$. Thus $[F]$ is a non-negative, integral matrix.

If we additionally have that $\overline{\mathbf{M}}(F)$ is exact, then we also have the matrix $\llbracket F\rrbracket$ whose rows and 
columns are indexed by isomorphism classes of simple objects in $\overline{\mathcal{M}}$ and the $X \times Y$-entry is give the 
composition multiplicity of $X$ in $\overline{\mathbf{M}}(F)\,Y$.

If $(F,G)$ is an adjoint pair of $1$-morphisms, then $\overline{\mathbf{M}}(G)$ is exact and $[F]^t = \llbracket G \rrbracket$, see
\cite[Lemma 10]{MM5}.

\subsection{The algebra $\Lambda_n$}\label{s2.5}
For a positive integer $n$, the \emph{star} $S_n$ on $n + 1$ vertices is the graph 
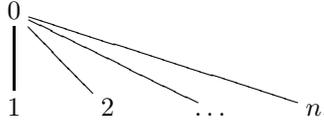
\begin{displaymath}
 \xymatrix{
 0\ar@{-}[d]\ar@{-}[dr]\ar@{-}[drr]\ar@{-}[drrr] \\
 1 & 2 & \dots & n.
 }
\end{displaymath}
Let $\overline{S_n}$ be the double quiver of $S_n$, i.e.,\ the quiver where 
each edge $\{0,i\}$ in the underlying graph $S_n$ is replaced by arrows $a_i := (0,i), b_i := (i,0)$. 
For examples the quiver of $\overline{S_1}$ looks like
 
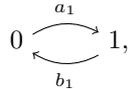
\[
  \begin{tikzpicture}
  \matrix(m)[matrix of math nodes, row sep=2.5em, column sep=2.5em,
    text height=1.0ex, text depth=0.25ex]
    {0 & 1, \\};
    \path[->,font=\scriptsize]
    (m-1-1) edge[bend left] node[above] {$a_1$} (m-1-2)
    (m-1-2) edge[bend left] node[below] {$b_1$} (m-1-1);
    \end{tikzpicture}
\]
and for $\overline{S}_2$ we have
 
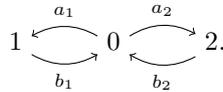
\[
  \begin{tikzpicture}
  \matrix(m)[matrix of math nodes, row sep=2.5em, column sep=2.5em,
    text height=1.0ex, text depth=0.25ex]
    {1 & 0 & 2. \\};
    \path[->,font=\scriptsize]
    (m-1-1) edge[bend right] node[below] {$b_1$} (m-1-2)
    (m-1-2) edge[bend right] node[above] {$a_1$} (m-1-1)
    (m-1-3) edge[bend left] node[below] {$b_2$} (m-1-2)
    (m-1-2) edge[bend left] node[above] {$a_2$} (m-1-3);
    \end{tikzpicture}
\]

Let $\Bbbk \overline{S_n}$ be the path algebra of $\overline{S_n}$.
We now define the algebra $\Lambda = \Lambda_n$ as the quotient of $\Bbbk \overline{S_n}$ by certain relations 
(which depend on $n$). In the case $n = 1$ we require that
\begin{itemize}
 \item $b_1a_1b_1 = a_1b_1a_1 = 0$.
\end{itemize}
If $n > 1$, we require that
\begin{itemize}
 \item for all $1 \leq i,j \leq n$ we have $b_ia_i = b_ja_j$;
 \item for all $1 \leq i < j \leq n$ we have $a_jb_i = a_ib_j = 0$.
\end{itemize}
Note that these relations imply that for all $1 \leq i \leq n$ we have $a_ib_ia_i = b_ia_ib_i = 0$.
For $0 \leq i \leq n$, we denote by $e_i$ the idempotent of $\Lambda$ which corresponds to the vertex $i$. 
We set $P_i := \Lambda e_i$ and denote the simple top of $P_i$ by $L_i$.

The structure of projective $\Lambda$-modules follows directly from the defining re\-la\-tions:
\begin{itemize}
\item If $1 \leq i \leq n$, then $P_i$ is a uniserial pro\-jec\-tive-in\-jec\-tive module of Loe\-wy
length three. Its top and so\-cle are isomorphic to $L_i$ and the module $\mathrm{Rad}(P_i)/\mathrm{Soc}(P_i)$
is iso\-mor\-phic to the simple $L_0$.
\item The projective $P_0$ is also injective of Loewy length three with  top and socle isomorphic to $L_0$. 
The module $\mathrm{Rad}(P_0)/\mathrm{Soc}(P_0)$ is isomorphic to the multiplicity free direct sum of all 
$L_i$ for $1 \leq i \leq n$.
\end{itemize}
From the above description we see that the algebra $\Lambda$ is self-injective.

\subsection{The $2$-category of projective $\Lambda$-$\Lambda$-bimodules}\label{s2.6}
Let $n \geq 1$ and $\Lambda = \Lambda_n$ be the algebra defined in Section \ref{s2.5}. Moreover, let $\mathcal{C}$ 
be a small category equivalent to $\Lambda$-proj. Then we define the $2$-category $\cC_{\Lambda, \mathcal{C}} = \cC_{\Lambda}$ 
to be the $2$-category with  
\begin{itemize}
 \item one object $\mathtt{i}$ which we identify with $\mathcal{C}$;
 \item $1$-morphisms which are those endofunctors of $\mathcal{C}$ that are isomorphic to tensoring with $\Lambda$-$\Lambda$-bimodules in 
 $\on{add}(\Lambda \oplus (\Lambda\otimes_{\Bbbk} \Lambda))$;
 \item $2$-morphisms being all natural transformations of functors.
\end{itemize}
Note that the indecomposable $1$-morphisms in $\cC_{\Lambda}$ are, up to isomorphism, the identity $\mathbb{1}_{\mathtt{i}}$ 
which is isomorphic to tensoring with the $\Lambda$-$\Lambda$-bimodule $_\Lambda \Lambda_{\Lambda}$ and
\[
 F_{ij} \simeq \Lambda e_i \otimes e_j\Lambda \otimes_\Lambda {}_-, \quad \text{for $0 \leq i,j \leq n$.}
\]
The cell structure of $\cC_{\Lambda}$ is described as follows.
For each $0 \leq i \leq n$ there is a left cell $\mathcal{L}_i := \setof{F_{ji}}{0 \leq j \leq n}$, and a 
right cell $\mathcal{R}_i = \setof{F_{ij}}{0 \leq j \leq n}$. The union of all $\mathcal{L}_i$ forms the 
maximal two-sided cell $\mathcal{J}$ of $\cC_{\Lambda}$. Moreover, there is the left, right and two-sided cell 
$\mathcal{J}_e$ consisting of the identity $1$-morphism $\mathbb{1}_{\mathtt{i}}$.

The main object of study in the present paper is the $2$-full $2$-subcategory $\cC = \cC_n$ of $\cC_{\Lambda_n}$ given by
the additive closure of $1$-morphisms $\mathbb{1}_{\mathtt{i}}, F_{00}, F_{10}, \ldots, F_{n0}$. We will loosely refer to 
$\cC_n$ as a \emph{left cell} $2$-subcategory of $\cC_{\Lambda}$. Our main problem 
is to classify simple transitive $2$-representations of $\cC$.

Now, let $n \geq 1$ and $\mathbf{M}$ be a simple transitive $2$-representation of $\cC$ of rank $r$. There are, up to 
isomorphism and apart from the identity, $n+1$ indecomposable, pairwise non-isomorphic $1$-morphisms in $\cC$. 
Their composition is given by
\begin{align}\label{eq:1}
 F_{i0} \circ F_{j0} \simeq \begin{cases}
			F_{i0}^{\oplus 2}, \quad &\text{ if } j = 0;\\
					F_{i0}, \quad &\text{ else.}
                       \end{cases} 
\end{align}

From \eqref{eq:1} we can deduce that $\cC$ has one left cell $\l$ which does not contain $\mathbb{1}_{\mathtt{i}}$.
The cell $\l$ consists of all $F_{i0}$ and is a two-sided cell at the same time. Inside $\l$ there are $n + 1$
distinct right cells each consisting of exactly one of the $F_{i0}$.

For our study of simple transitive $2$-representations of $\cC$ we denote by $B$ a basic $\Bbbk$-algebra such that 
$\mathbf{M}(\mathtt{i})$ is equivalent to $B$-proj. Moreover, we let $1 = \epsilon_0 + \cdots + \epsilon_{r-1}$ be
a decomposition of the identity in $B$ into a sum of pairwise orthogonal primitive idempotents. In the same
spirit as for $\Lambda$, we denote by $G_{ij}$ the endofunctor of $\mathbf{M}(\mathtt{i})$ isomorphic to tensoring 
with the projective $B$-$B$-bimodule $B\epsilon_i \otimes \epsilon_jB$. Moreover, for $0 \leq i \leq r-1$, we denote 
by $Q_i$ the projective $B$-module $B\epsilon_i$ and by $\hat{L}_i$ the simple top of $Q_i$.

Now, we set 
\[
 F := \bigoplus_{i=0}^{n}F_{i0}.
\]

Due to the transitivity of $\mathbf{M}$, we have that the matrix $\mathsf{M} := [F]$ of $F$ must be irreducible.
Moreover, we have $\mathsf{M} = (n+2)\mathsf{M}$, as can be read off from \eqref{eq:1}. 
All such matrices are, up to a permutation of basis vectors, classified by \cite[Lemma~4.3]{TZ} which we will 
make use of later.

For now, we just note that we have the following lemma:
\begin{lemma}\label{lem1}
Let $\mathbf{M}$ be a transitive $2$-representation of $\cC$. 
Then there is an ordering of indecomposable objects in $\mathbf{M}(\mathtt{i})$ such that 
\begin{displaymath}
[F_{00}]=
\left(\begin{array}{c|c}2E&*\\\hline 0&0\end{array}\right),
\end{displaymath}
where $E$ is the identity matrix.
\end{lemma}

\begin{proof}
Mutatis mutandis the proof of \cite[Lemma 5.3]{Zi1}. 
\end{proof}

\subsection{Non-negative idempotent matrices}\label{s2.7}
Additionally, we will need the following classification result for non-negative idempotent matrices by Flor; see \cite{F69}.
\begin{theorem}
\label{Flor}
 Let $I$ be a non-negative real idempotent matrix of rank $k$. Then there exists a permutation
 matrix $P$ such that
 \[
 P^{-1}IP = 
   \begin{pmatrix}
    0 & AJ & AJB\\
    0 & J & JB \\
    0 & 0 & 0 \\
   \end{pmatrix}
\quad\text{ with }\quad
 J =  
  \begin{pmatrix}
    J_1 &  0   & \cdots & 0 \\
    0   & J_2 &  &  \vdots\\
    \vdots & & \ddots & 0\\
    0 & \cdots & 0 & J_k
  \end{pmatrix}.
\]
Here, each $J_i$ is a non-negative idempotent matrix of rank one and $A, B$ are non-negative matrices
of the appropriate size.
\end{theorem}

\begin{remark}
This theorem can be applied to quasi-idempotent (but not nilpotent) matrices as well. 
If $I^2 = \lambda I$ and $\lambda\neq 0$, then $(\frac{1}{\lambda}I)^2 = \frac{1}{\lambda^2}I^2 = 
\frac{1}{\lambda}I$. Hence $\frac{1}{\lambda}I$ is idempotent and thus can be described by the above theorem. 
\end{remark}

\section{$1$-morphisms act as projective functors}\label{s3}
In this section we generalize a result which states that the action of $1$-morphisms on a simple transitive $2$-representation 
sends non-zero objects to projective objects, moreover, the action is given by projective functors; cf.\ \cite[Theorem 11]{KMMZ}. 
The above statement is proved for fiat $2$-categories. We show that this result, under some mild assumptions, can be extended 
from the fiat case to finitary $2$-categories which fits our setup. Our proof 
follows closely the one provided in \cite{KMMZ} with a few minor adjustments.

\begin{theorem}\label{thmprojfunc}
Let $\mathbf{M}$ be a simple transitive $2$-representation of a finitary $2$-ca\-te\-go\-ry $\mathscr{D}$.
Let, further,  $\mathcal{I}$ be the apex of $\mathbf{M}$ and $\mathrm{F}\in \mathcal{I}$. 
Then the following holds:
\begin{enumerate}[$($i$)$]
  \item\label{thmprojfunc.1} For every object $X$ in any $\overline{\mathbf{M}}(\mathtt{i})$, 
    the object $\mathrm{F}\,X$ is projective. 
  \item\label{thmprojfunc.2} If $\overline{\mathbf{M}}(F)$ is additionally left exact, 
  then the functor $\overline{\mathbf{M}}(\mathrm{F})$ is a projective functor.
\end{enumerate}
\end{theorem}

\begin{proof}
Without loss of generality we may assume that $\mathcal{I}$ is the maximum
two-sided cell of $\mathscr{D}$.

Denote by $Q$ the complexification of the split Grothendieck group of
\begin{displaymath}
 \coprod_{\mathtt{i}\in\mathscr{D}}{\mathbf{M}}(\mathtt{i}). 
\end{displaymath}
Let $\mathbf{B}$ denote the distinguished basis in $Q$ given by classes of indecomposable modules.
Let $\mathrm{F}_1,\mathrm{F}_2,\dots,\mathrm{F}_k$ be a 
complete and irredundant list of indecomposable $1$-morphisms in $\mathcal{I}$. For $i=1,2,\dots,k$, 
let $\mathcal{X}^{(i)}_{\bullet}$ be a minimal projective presentation of $\mathrm{F}_i\, X$. Further,
for $j\geq 0$, denote by $v_j^{(i)}$ the image of $\mathcal{X}^{(i)}_{j}$ in $Q$.
Note that  $v_j^{(i)}$ is a linear combination of elements in $\mathbf{B}$
with non-negative integer coefficients.

As $\mathscr{D}$ is assumed to be idempotent, there are $i,j\in\{1,2,\dots,k\}$ such that 
$\mathrm{F}_i\circ \mathrm{F}_j\neq 0$. Therefore we may apply \cite[Proposition~18]{KM2}, 
which asserts that the algebra $A_{\mathscr{D}}$ has a unique idempotent $e$ of the form 

\begin{displaymath}
  \sum_{i=1}^kc_i[\mathrm{F}_i], 
\end{displaymath}

where all $c_i\in\mathbb{R}_{>0}$ and $[\mathrm{F}_i]$ denotes the image of $\mathrm{F}_i$ in $A_{\mathscr{D}}$.
The vector space $Q$ is, naturally, an $A_{\mathscr{D}}$-module. 

Let $j\geq 0$ and set

\begin{displaymath}
  v(j):=\sum_{i=1}^kc_iv_j^{(i)}.
\end{displaymath}

Applying any $1$-morphism $\mathrm{H}$ to a projective presentation of some object $Y$
and taking into account that the action of $\mathrm{H}$ is right exact, gives a
projective presentation of $\mathrm{H}\,Y$ which, a priori, does not have to be minimal.
By construction and minimality of $\mathcal{X}^{(i)}_{\bullet}$, this implies that  
$e(v(j))-v(j)$ is a  linear combination of elements in $\mathbf{B}$ with non-negative 
real coefficients. Note that transitivity of $\mathbf{M}$ implies that, for any linear combination
$z$ of elements in $\mathbf{B}$ with non-negative real coefficients, the element $ez$ is also
a linear combination of elements in $\mathbf{B}$ with non-negative real coefficients, moreover,
if $z\neq 0$, then $ez\neq 0$.
Therefore the equality $e^2=e$ yields $e(v(j))=v(j)$. Consequently, 
the projective presentation $\mathrm{F}_i\mathcal{X}^{(l)}_{\bullet}$ of $\mathrm{F}_i(\mathrm{F}_l\, X)$
is minimal, for all $i$ and $l$. 

Now the proof is completed by standard arguments as in \cite[Lemma~12]{MM5}. From the previous paragraph
it follows that homomorphisms $\mathcal{X}^{(l)}_{1}\to \mathcal{X}^{(l)}_{0}$,
for $l=1,2,\dots,k$, generate a $\cC$-invariant  ideal of $\mathbf{M}$ different from $\mathbf{M}$. 
Because of simple transitivity of $\mathbf{M}$, this ideal must be zero.
Therefore all homomorphisms $\mathcal{X}^{(l)}_{1}\to \mathcal{X}^{(l)}_{0}$ are zero which implies
$\mathcal{X}^{(l)}_{0}\simeq \mathrm{F}_l\, X$. Claim~\eqref{thmprojfunc.1} follows.
Claim~\eqref{thmprojfunc.2} follows from claim~\eqref{thmprojfunc.1} and \cite[Lemma~13]{MM5}. Indeed,
$F$ is right exact by construction and left exact by assumption, hence exact which allows us to apply
\cite[Lemma~13]{MM5}. This completes the proof.
\end{proof}

\begin{remark}\label{rmk:projfunc}
For any simple transitive $2$-representation $\mathbf{M}$ of our $2$-category $\cC$, Theorem~\ref{thmprojfunc} implies 
that all $1$-morphisms of $\cC$ send any object to a projective object. The $1$-morphism $F_{00}$ is self-adjoint and hence 
acts as a projective functor. Now, \cite[Lemma 8]{MZ1} yields that all $F_{i0}$ act as projective functors.
\end{remark}

\section{The case $n = 1$}\label{s4}
Before we study the general case, let us consider the case where $n = 1$. The aim of this section is to prove 
the following theorem:
\begin{theorem}
\label{SimpleTransitiveIsCell2x2}
 Let $\mathbf{M}$ be a simple transitive $2$-representation of $\cC_1$. Then $\mathbf{M}$ is equivalent to 
 a cell $2$-representation of $\cC$.
\end{theorem}

For this section we let $\mathbf{M}$ be a simple transitive $2$-representation of $\cC_1$ of rank $r$ and 
$\mathbf{M}(\mathtt{i}) \simeq B\text{-proj}$ where $B$ is as in Section \ref{s2.5}.

Moreover, we recall that, for $n = 1$, there are, up to isomorphism, $3$ distinct 1-morphisms in $\cC$, 
namely $\mathbb{1}_{\mathtt{i}}, F_{00}, F_{10}$; and the non-trivial part of the composition table looks as follows:
\begin{equation}
\label{eq:10}
\begin{array}{c||l|l}
\circ& F_{00} & F_{10}  \\
\hline
F_{00}	& F_{00} \oplus F_{00} & F_{00} \\
\hline
F_{10}	& F_{10} \oplus F_{10} & F_{10}
\end{array}
\end{equation}

\subsection{An analysis of the possible matrices}\label{s4.1}
The main goal of this section is to prove the following result for the matrix of $[F]$.
\begin{lemma}
\label{MatrixLemma2x2}
 Let $\mathbf{M}$ be a simple transitive $2$-representation of $\cC_1$ with apex $\l = \{F_{00}, F_{10}\}$. 
 Then the matrix $\mathsf{M}$ of the action of $F$ is given by
 \begin{align*}
  \mathsf{M} = \begin{pmatrix}
       2 & 1 \\
       2 & 1
      \end{pmatrix}.
 \end{align*}
 Moreover, the matrices $\mathsf{M}_0$ and $\mathsf{M}_1$ of the action of $F_{00}$ and $F_{10}$, respectively, are
 \begin{align*}
  \mathsf{M}_0 = \begin{pmatrix}
         2 & 1\\
         0 & 0
        \end{pmatrix}, \quad
  \mathsf{M}_1 = \begin{pmatrix}
         0 & 0 \\
         2 & 1
        \end{pmatrix}.
 \end{align*}
\end{lemma}

\begin{proof} \label{ProofMatrixLemma2x2}
As mentioned above, the matrix $\mathsf{M} := [F]$ must be irreducible and all $\mathbf{M}(F_{i0})$ are non-zero, 
since the apex of $\mathbf{M}$ is $\l$ and $\mathbf{M}$ is simple transitive. Moreover, we have $\mathsf{M} = 3\mathsf{M}$, 
as can be read off from \eqref{eq:10}. From \cite[Lemma~4.3]{TZ} it follows that, up to a permutation of basis vectors, 
we have that $\mathsf{M}$ has to be contained in the following set of matrices:

\begin{align*}
 \mathsf{M} \in 
 \left\{(3), \begin{pmatrix}
               2 & 1\\
               2 & 1
              \end{pmatrix},
              \begin{pmatrix}
               2 & 2\\
               1 & 1
              \end{pmatrix},
              \begin{pmatrix}
               1 & 1 & 1\\
               1 & 1 & 1\\
               1 & 1 & 1
              \end{pmatrix}\right\}.
\end{align*}
Since $\mathsf{M} = \mathsf{M}_0 + \mathsf{M}_1$, Lemma \ref{lem1} immediately excludes the $3 \times 3$ case. 

Now let us assume that $\mathsf{M} = (3)$. In this case we have that $B$-$\on{proj}$ has only one projective indecomposable $B\epsilon_0$. 
By Remark \ref{rmk:projfunc}, we have that $F_{00}$ and $F_{10}$
act as projective functors, i.e.,
\[
 \mathbf{M}(F_{00}) = G_{00}^{\oplus a}, \quad \mathbf{M}(F_{10}) = G_{00}^{\oplus b}.
\]
From Lemma \ref{lem1}, we get that $[F_{00}] = (2)$ and, consequently, $[F_{10}] = (1)$. Consider the $2$-full $2$-subcategory $\mathscr{D}$ of 
$\mathscr{C}$ generated by the additive closure of $\mathbb{1}_{\mathtt{i}}$ and $F_{00}$ and the restriction of the action of 
$\mathbf{M}$ to $\mathscr{D}$. Since the endomorphism algebra of $F_{00}$ is isomorphic to the dual numbers $D := \Bbbk[x] \slash (x^2)$, it follows that 
$\mathscr{D}$ is biequivalent to $\cC_{D}$. This means that $F_{10}$ has to act as a projective functor 
over the dual numbers which contradicts the fact that $[F_{10}] = (1)$. Hence, this case is not possible.

Next, we assume that $\mathsf{M} = \begin{pmatrix}2 & 2\\1 & 1\end{pmatrix}$. By construction, we have that 
$F = F_{00} \oplus F_{10}$ and thus $\mathsf{M} = [F] = [F_{00}] + [F_{10}]$. 
Moreover, by Lemma \ref{lem1} and Theorem \ref{Flor} we have that $\mathsf{M}_0 = [F_{00}]$ has to be of the form
\[
 \mathsf{M}_0 =\begin{pmatrix}
          2 & a\\
          0 & 0
         \end{pmatrix}
\]
for some $a$, which implies that
\[
 \mathsf{M}_1 = \begin{pmatrix}
          0 & b\\
          1 & 1
         \end{pmatrix},
\]
for some $b$. Moreover, since $F_{10}$ is an idempotent $1$-morphism it follows that $b = 0$ and $a = 2$. Therefore
\[
 \mathsf{M}_0 = \begin{pmatrix} 2 & 2\\0 & 0\end{pmatrix}, \quad \mathsf{M}_1 = \begin{pmatrix} 0 & 0\\1 & 1\end{pmatrix}.
\]

From this we can see that $F_{00}$ sends both $Q_0$ and $Q_1$ to two copies of $Q_0$ and $F_{10}$ 
sends both projectives to $Q_1$, i.e.,
\begin{align*}
 \overline{\mathbf{M}}(F_{00}) = G_{00}^{\oplus a_{00}} \oplus G_{01}^{\oplus a_{01}}, \quad
 \overline{\mathbf{M}}(F_{10}) = G_{10}^{\oplus a_{10}} \oplus G_{11}^{\oplus a_{11}}. 
\end{align*}

Now, since $F_{00}$ is self-adjoint, we have that $\mathsf{M}_0^{\text{tr}} = \begin{pmatrix} 2 & 0\\ 2 & 0\end{pmatrix}$ 
is the matrix of the action of $F_{00}$ in the basis of simples, as was proved in \cite[Lemma 10]{MM5}. 
Since we know the matrix of $F_{00}$ in the basis of simples, we obtain that $\overline{\mathbf{M}}(F_{00})$ annihilates 
$\hat{L}_1$ which implies that $a_{01} = 0$ because 
\[
  G_{01}^{\oplus a_{01}}\hat{L}_1 = \left(B\epsilon_0 \otimes \epsilon_1B\right)^{\oplus a_{01}}\otimes_B \hat{L}_1 
  \simeq B\epsilon_0^{\oplus a_{01}}.
\]

Additionally, we know that 
\[
 \overline{\mathbf{M}}(F_{10})^{\oplus 2}\hat{L}_1 = \overline{\mathbf{M}}(F_{10}\circ F_{00})\hat{L}_1 
 = \overline{\mathbf{M}}(F_{10})\circ \overline{\mathbf{M}}(F_{00})\hat{L}_1 = 0.
\]
Since $\hat{L}_1$ is a subquotient of $Q_1$ and the latter is a summand of $\overline{\mathbf{M}}(F_{00})\hat{L}_1$, we 
obtain that $F_{10}$ also annihilates $\hat{L}_1$. Therefore, we have 
\begin{align*}
 \overline{\mathbf{M}}(F_{00}) = G_{00}^{\oplus a_{00}}, \quad
 \overline{\mathbf{M}}(F_{10}) = G_{10}^{\oplus a_{10}}. 
\end{align*}

Furthermore, we know by Theorem \ref{thmprojfunc} that $\overline{\mathbf{M}}(F_{00})\hat{L}_0$ is projective 
which means it is a direct sum of a number of copies of $Q_0$. As $\overline{\mathbf{M}}(F_{00})$ annihilates 
$\hat{L}_1$ and is self-adjoint, we obtain that 
\[
 \on{dim}(\on{Hom}(\hat{L}_1, \overline{\mathbf{M}}(F_{00})\hat{L}_0)) = \on{dim}(\on{Hom}(\overline{\mathbf{M}}(F_{00})
 \hat{L}_1, \hat{L}_0)) = 0
\]
and thus $\hat{L}_1$ cannot appear in the socle of $\mathbf{M}(F_{00})\hat{L}_0$. By the above, $\hat{L}_1$
does not appear in the top of $\mathbf{M}(F_{00})\hat{L}_0$ either as $\mathbf{M}(F_{00})\hat{L}_0 \simeq Q_0^{\oplus a_{00}}$. 
On the other hand, from $\mathsf{M}_0^{\text{tr}}$ it follows that $\hat{L}_0$ has multiplicity $2$ in 
the Jordan-H\"{o}lder series of $\overline{\mathbf{M}}(F_{00})\hat{L}_0$ which means that $a_{00}$ is either equal 
to $1$ or $2$. If $a_{00} = 2$, then both $\hat{L}_0$ and $\hat{L}_1$ have multiplicity $1$ in $Q_0$. This means 
that the socle of $Q_0$ has to be isomorphic to $\hat{L}_1$ which, by the above, cannot be 
the case. It follows that $a_{00} = 1$. Therefore $\overline{\mathbf{M}}(F_{00})\hat{L}_0$ has to be isomorphic 
to $Q_0$ and has simple top and socle isomorphic to $\hat{L}_0$. This also yields that
\begin{displaymath}
  \on{dim}(\epsilon_0B\epsilon_0) = \on{dim}(\on{Hom}(Q_0, Q_0)) = 2. 
\end{displaymath}
Now the composition of $\overline{\mathbf{M}}(F_{00})$ and $\overline{\mathbf{M}}(F_{10})$ gives us that
\begin{align*}
 G_{00} = \overline{\mathbf{M}}(F_{00}) & = \overline{\mathbf{M}}(F_{00} \circ F_{10}) = 
 \overline{\mathbf{M}}(F_{00}) \circ \overline{\mathbf{M}}(F_{10}) = G_{00} \circ G_{10}^{\oplus a_{10}} \\
 & = G_{00}^{\oplus a_{10} \cdot \on{dim}(\epsilon_0B\epsilon_1)}.
\end{align*}
This yields that $a_{01} = \on{dim}(\epsilon_0B\epsilon_1) = 1$; in other words, we have 
$\overline{\mathbf{M}}(F_{10}) = G_{10}$. However, we know that the projective $Q_0$ is 
mapped to $Q_1$ by $\overline{\mathbf{M}}(F_{10})$ but 
\[
 G_{10}Q_0 = B\epsilon_1 \otimes \epsilon_0B \otimes_B B \epsilon_0 \simeq 
 (B\epsilon_1)^{\oplus \on{dim}(\epsilon_0B\epsilon_0)} =  Q_1^{\oplus 2},
\]
which is a contradiction.

We have thus proved that $\mathsf{M} = \begin{pmatrix} 2 & 1 \\ 2 & 1\end{pmatrix}$ is the only possibility 
and one checks easily that this is realized by the cell $2$-representation $\mathbf{C}_{\l}$. By Lemma~\ref{lem1}
and Theorem~\ref{Flor} we have that $\mathsf{M}_0 = \begin{pmatrix} 2 & a \\ 0 & 0\end{pmatrix}$
and thus $\mathsf{M}_1 = \begin{pmatrix} 0 & b \\ 2 & 1\end{pmatrix}$. Again, the fact that $\mathsf{M}_1$ 
has to be idempotent yields $b = 0$ and thus $a = 1$. Therefore
\begin{align*}
 \mathsf{M}_0  = \begin{pmatrix} 2 & 1\\0 & 0\end{pmatrix}, \quad
 \mathsf{M}_1  = \begin{pmatrix} 0 & 0\\2 & 1\end{pmatrix}.
\end{align*}
\end{proof}

\subsection{Proof of Theorem \ref{SimpleTransitiveIsCell2x2}}\label{s4.2}
Now, we are ready to prove Theorem \ref{SimpleTransitiveIsCell2x2}. 

Let $\mathbf{M}$ be a simple transitive $2$-representation of $\cC_1$. If $\mathbf{M}(F_{i0}) = 0$ for 
either $i = 0$ or $i = 1$, then the composition table implies that both $F_{00}$ and $F_{10}$ act as $0$ and thus 
$\mathcal{J}_e$ is the apex of $\mathbf{M}$. If we factor out the other cell, we obtain a 
$2$-representation of a $2$-category with only one $1$-morphism, and thus, in this case, $\mathbf{M}$ has to 
be equivalent to a cell $2$-representation by \cite[Theorem~$18$]{MM5}. Therefore, we may assume that 
$\mathbf{M}$ is $1$-faithful, i.e.,\ that all $\mathbf{M}(F_{i0})$ are non-zero.

Now consider $\overline{\mathbf{M}}$. Then the matrices $\mathsf{M}_0$ and $\mathsf{M}_1$ are given by 
Lemma~\ref{MatrixLemma2x2}. From these matrices it follows that $F_{00}$ annihilates $\hat{L}_1$. Since 
$F_{10} \oplus F_{10} \simeq F_{10} \circ F_{00}$, it follows that $F_{10}$ annihilates $\hat{L}_1$ as well. 
Moreover, by self-adjointness of $F_{00}$, we know that the action of $F_{00}$ in the basis of simples is given by 
\begin{displaymath}
 \mathsf{M}_0^{\text{tr}} = \begin{pmatrix} 2 & 0\\ 1 & 0 \end{pmatrix}.
\end{displaymath}

The projective module $\overline{\mathbf{M}}(F_{00})\hat{L}_0$ is isomorphic to a direct sum of a number of copies of $Q_0$.
As $\overline{\mathbf{M}}(F_{00})\hat{L}_0$ has the simple subquotient $\hat{L}_1$ with multiplicity $1$, it follows that
$\overline{\mathbf{M}}(F_{00})\hat{L}_0 \simeq Q_0$. 

Finally, we have that 
\begin{align*}
 (\overline{\mathbf{M}}(F_{10})\hat{L}_0)^{\oplus 2} = \overline{\mathbf{M}}(F_{10}\circ F_{00}) \hat{L}_0 = 
 \overline{\mathbf{M}}(F_{10}) Q_0 = Q_1^{\oplus 2},
\end{align*}
and thus we obtain
\begin{IEEEeqnarray*}{c;c;c;c;c;c;c}
 \overline{\mathbf{M}}(F_{00})\hat{L}_0 & \simeq & Q_0, & \quad & \overline{\mathbf{M}}(F_{10})\hat{L}_0 & \simeq &  Q_1, \\
 \overline{\mathbf{M}}(F_{00})\hat{L}_1 & \simeq & 0, & \quad & \overline{\mathbf{M}}(F_{10})\hat{L}_1 & \simeq & 0.
\end{IEEEeqnarray*}

Now we define the following map 
\begin{align*}
 \Phi: \mathbf{P}_{\mathtt{i}} &\to \overline{\mathbf{M}},\\
 \mathbb{1}_{\mathtt{i}} & \mapsto \hat{L}_0.
\end{align*}
From the above analysis we see that $\Phi$ maps $\mathbf{N} = \on{add}(\setof{F}{F \in \mathcal{L}})$ to the $2$-re\-pre\-sen\-ta\-ti\-on 
$\overline{\mathbf{M}}_{\text{proj}}$ of $\cC_1$ given by projective objects in $\overline{\mathbf{M}}$. Note that $\overline{\mathbf{M}}_{\text{proj}}$ 
is equivalent to $\mathbf{M}$. We even have a bijection between the isomorphism classes of indecomposable objects in $\mathbf{N}$
and in $\overline{\mathbf{M}}_{\text{proj}}$. To complete the proof, we need to show that $\Phi$ induces isomorphisms
between the homomorphism spaces of the corresponding indecomposable objects. For this it is enough to show that every simple 
transitive $2$-representation of $\cC$ has the same Cartan matrix. 

From our analysis of $\overline{\mathbf{M}}(F_{00})\hat{L}_0 = Q_0$ we know that 
\begin{align*}
 \on{dim}(\operatorname{Hom}(Q_0,Q_0)) = 2, \quad \on{dim}(\operatorname{Hom}(Q_1,Q_0)) = 1,
\end{align*}
Moreover, using the fact that $F_{00}$ is self-adjoint, we can deduce that 
\begin{align*}
 \on{dim}(\operatorname{Hom}(Q_0,Q_1)) & = \on{dim}(\operatorname{Hom}(F_{00}\hat{L}_0,Q_1)) \\
 & = \on{dim}(\operatorname{Hom}(\hat{L}_0,F_{00}Q_1)) \\
 &= \on{dim}(\operatorname{Hom}(\hat{L}_0, Q_0)) = 1.
\end{align*}

Now consider $\on{dim}(\operatorname{Hom}(Q_1,Q_1)) \geq 1$. From the fact that 
$\on{dim}(\on{Hom}(Q_0,Q_1)) = 1$ we obtain that $\hat{L}_0$ has multiplicity $1$ in the Jordan-H\"{o}lder 
series of $Q_1$. Therefore there is an extension from $\hat{L}_1$ to $\hat{L}_0$. Let $K$ be a corresponding
indecomposable module with top $\hat{L}_1$ and socle $\hat{L}_0$. Then we have a short exact sequence
\begin{displaymath}
 0 \to N \to Q_1 \to K \to 0,
\end{displaymath}
where the Jordan-H\"{o}lder series of $N$ only contains $\hat{L}_1$. However, as both $F_{00}$ and $F_{10}$ 
annihilate $\hat{L}_1$, the module $N$ 
generates a $\cC$-stable ideal of $\overline{\mathbf{M}}_{\on{proj}}(\mathtt{i})$, not containing any identity $2$-morphisms.
This contradicts the fact that $\mathbf{M}$ is simple transitive. 
Consequently, we get
\begin{align*}
 \on{dim}(\operatorname{Hom}(Q_1,Q_1)) = 1,
\end{align*}
and hence the Cartan matrix of $\mathbf{M}$ in the basis $Q_0, Q_1$ equals 
\begin{align*}
 \begin{pmatrix}
  2 & 1 \\
  1 & 1
 \end{pmatrix}.
\end{align*}
This coincides with the Cartan matrix of the cell $2$-representation and implies that $\Phi$ induces an equivalence
from the cell $2$-representation to $\mathbf{M}$. This concludes the proof. 

\section{The general case}\label{s5}
Let $n \geq 1$. In the general case we were not able to classify all simple transitive $2$-representations of $\cC_n$
but we believe that the following is true:
\begin{conjecture}
\label{GeneralCaseConjecture}
 Equivalence classes of simple transitive $2$-representations of $\cC_n$ are in bijection with set partitions of $\{1, 2, \ldots, n\}$. 
\end{conjecture}

In the remainder of the section, we present the evidence which we have for this conjecture which will also explain the content
of the conjecture in more detail.
\subsection{An analysis of the possible matrices}\label{s5.1}
The goal of this section is to prove an analogue of Lemma \ref{MatrixLemma2x2}. First, recall that
\[
 F := \bigoplus_{i=0}^n F_{i0}.
\]
Moreover, let $\mathsf{M} = [F]$ be the matrix of the action of $F$, and $\mathsf{M}_i = [F_{i0}]$ the matrix of the 
action of $F_{i0}$. 
\begin{theorem}
\label{MatrixThmGeneral}
Let $\mathbf{M}$ be a simple transitive $2$-representation of $\cC$ of rank $r$ with apex $\l$. Then there exists an ordering
of the isomorphism classes of indecomposable objects in $\mathbf{M}(\mathtt{i})$ and a surjective function $\varphi: \{1, 2, \ldots, n\}
\to \{2, 3, \ldots, r\}$ such that 
\begin{align}
\label{eq:MatrixCase1}
 \mathsf{M}_0 = \begin{pmatrix}
              2 & 1 & \cdots & 1\\
              0 & 0 & \cdots & 0\\
              \vdots & \vdots & \ddots & \vdots \\
              0 & 0 & \cdots & 0
             \end{pmatrix},  \quad
  \mathsf{M}_i = \begin{pmatrix}
              0 & 0 & \cdots & 0\\
              \vdots & \vdots &\ddots & \vdots \\
              0 & 0 & \cdots & 0\\
              2 & 1 & \cdots & 1 \\
              0 & 0 & \cdots & 0\\
              \vdots & \vdots & \ddots & \vdots \\
              0 & 0 & \cdots & 0
             \end{pmatrix}, 
\end{align}
where the index of the non-zero row of $\mathsf{M}_i$ is $\varphi(i)$.
\end{theorem}
\begin{proof}
 Let $\mathbf{M}$ be as above. By the same arguments as in the proof of Lemma \ref{MatrixLemma2x2}, we may assume that $\mathbf{M}$ is 
 $1$-faithful, i.e.,\ that all $\mathbf{M}(F_{i0})$ are non-zero.

 As in the $n = 1$ case, we have that $\mathsf{M}$ has to be a positive, irreducible, integral matrix such that 
 $\mathsf{M}^2 = (n+2)\mathsf{M}$. From \cite[Proposition 4.1]{TZ} it follows that all such matrices have rank $1$ and trace $n+2$. 
 Furthermore, we have $\mathsf{M} = \mathsf{M}_0 + \mathsf{M}_1 + \cdots + \mathsf{M}_n$.
 
 We know from Lemma~\ref{lem1} that $\mathsf{M}_0$ has only $2$'s and $0$'s on the diagonal. Moreover $\mathsf{M}$ has trace $n+2$ 
 and is the sum of $n+1$ matrices which all are quasi-idempotent and non-zero. Thus each of them has a non-zero diagonal entry. 
 Hence it follows that $\mathsf{M}_0$ has exactly one $2$ on the diagonal. The other $\mathsf{M}_i$ have exactly one diagonal 
 entry equal to $1$ and all others equal to $0$.
 
 From \eqref{eq:1} we obtain that 
 \begin{align}
 \label{eq:MatrixMult}
  \mathsf{M}_i\mathsf{M}_j = a_j \mathsf{M}_i, \quad \text{ where } a_j = \begin{cases}
                                                                        2, & \quad j = 0,\\
                                                                        1, & \quad \text{else}.
                                                                       \end{cases}
 \end{align}
 
 This multiplication rule implies that no $\mathsf{M}_j$ has a zero column because if the $k$-th column of 
 $\mathsf{M}_j$ is zero, then so is the $k$-th column of $\mathsf{M}_i$, for all $0 \leq i \leq n$. Consequently, 
 the $k$-th column of $\mathsf{M}$ is zero which is impossible as $\mathsf{M}$ is a positive matrix.
 
 By Theorem \ref{Flor}, this implies that there exists a permutation matrix $P$ such that
 \begin{align*}
    P^{-1}\mathsf{M}_jP =  
    \begin{pmatrix}
      a_j & * & \cdots & *\\
      0 & 0 & \ldots & 0 \\
      \vdots & \vdots & \ddots& \vdots \\
      0 & 0 & \cdots  & 0\\
    \end{pmatrix}.
  \end{align*}
 Now, for each $2 \leq l \leq r$, pick $i_l\in \{1, 2, \ldots, n\}$ such that $\varphi(i_l) = l$ and set $i_1 = 0$. We define 
 \begin{displaymath}
  \mathsf{N} = \sum_{1 \leq l \leq r} \mathsf{M}_{i_l}.
 \end{displaymath}
 Then $\mathsf{N}$ is a positive $r \times r$-matrix such that $\mathsf{N}^2 = (r + 1)\mathsf{N}$ and such that the diagonal 
 of $\mathsf{N}$ is $(2, 1, \ldots, 1)$. 
 From \cite[Proposition~4.1]{TZ} we thus can conclude that $\mathsf{N} = (2, 1, \ldots, 1)^{\text{tr}}(1,1, \ldots, 1)$ or 
 $\mathsf{N} = (1, 1, \ldots, 1)^{\text{tr}}(2,1, \ldots, 1)$. For a matrix $\mathsf{M}_j$ which is not a summand of $\mathsf{N}$ we
 have to distinguish two cases. 
 
 Assume first that $\mathsf{M}_j$ is such that $\varphi(j) = l \neq 1$. Then we can
 replace $\mathsf{M}_{i_l}$ in the sum by $\mathsf{M}_j$. The result will have the same properties, i.e.,\ it will be quasi-idempotent
 and have the same trace. Since we did not change any of the other summands, we obtain that in this case $\mathsf{M}_j = \mathsf{M}_{i_l}$.
 
 Now assume that there exists $i_0 \neq 1$ such that the first row of $\mathsf{M}_{i_0}$ is non-zero,
 then 
 \begin{displaymath}
  \tilde{\mathsf{N}} = \mathsf{M}_{i_0} + \sum_{2 \leq l \leq r} \mathsf{M}_{i_l}
 \end{displaymath}
 has trace $(1,1, \dots, 1)$ and satisfies $\tilde{\mathsf{N}}^2 = r\tilde{\mathsf{N}}$. Then, by \cite[Proposition~4.1]{TZ}, 
 we obtain that $\tilde{\mathsf{N}}$ is the matrix with all entries equal to $1$. Again we did not change the other matrices
 and thus this case is only possible if $\mathsf{N}$ was of the form $\mathsf{N} = (2, 1, \ldots, 1)^{\text{tr}}(1,1, \ldots, 1)$.
 Thus we have established that the matrices $\mathsf{M}_i$ are either of the form 
 
 \begin{align*}
  \mathsf{M}_0 = \begin{pmatrix}
		2 & 1 & \cdots & 1\\
		0 & 0 & \cdots & 0\\
		0 & 0 & \cdots & 0
	      \end{pmatrix},  \quad
    \mathsf{M}_i = \begin{pmatrix}
		0 & 0 & \cdots & 0\\
		\vdots & \vdots &\ddots & \vdots \\
		0 & 0 & \cdots & 0\\
		2 & 1 & \cdots & 1 \\
		0 & 0 & \cdots & 0\\
		\vdots & \vdots & \ddots & \vdots \\
		0 & 0 & \cdots & 0
	      \end{pmatrix}, 
  \end{align*}
  or
  \begin{equation}\label{eq:matsecondcase}
  \mathsf{M}_0 = \begin{pmatrix}
		2 & 2 & \cdots & 2\\
		0 & 0 & \cdots & 0\\
		\vdots & \vdots & \ddots & \vdots \\
		0 & 0 & \cdots & 0
	      \end{pmatrix},  \quad
    \mathsf{M}_i = \begin{pmatrix}
		0 & 0 & \cdots & 0\\
		\vdots & \vdots &\ddots & \vdots \\
		0 & 0 & \cdots & 0\\
		1 & 1 & \cdots & 1 \\
		0 & 0 & \cdots & 0\\
		\vdots & \vdots & \ddots & \vdots \\
		0 & 0 & \cdots & 0
	      \end{pmatrix}.
  \end{equation}
  We note that some of the zero blocks in the matrices $\mathsf{M}_i$ may be empty, for example if the non-zero row
  is the first or last row.
  
  Next we want to prove that the case given by \eqref{eq:matsecondcase} cannot occur. The arguments here are 
  analogous to the arguments in the case of $n = 1$ for the corresponding set of matrices.

  Let $\mathbf{M}$ be a simple transitive $2$-representation of $\cC$ of rank $r$ with apex $\l$. 
  Consider the abelianization $\overline{\mathbf{M}}$. Assume that $\mathbf{M}$ is such that the matrices of actions of 
  the $F_{i0}$ are given as described in \eqref{eq:matsecondcase}. As in Section \ref{s4}, we let $B$ 
  be the basic, associative, unital $\Bbbk$-algebra such that $B$-mod is equivalent to 
  $\overline{\mathbf{M}}(\mathtt{i})$. By Remark~\ref{rmk:projfunc}, we know that all $F_{i0}$ act as
  projective functors, that is by tensoring with projective $B$-$B$-bimodules. 

  From the explicit form of the matrices given by \eqref{eq:matsecondcase} it is clear that 
  \begin{align*}
  \overline{\mathbf{M}}(F_{i0}) \simeq G_{\varphi(i)0}^{\oplus a_{\varphi(i)0}} \oplus G_{\varphi(i)1}^{\oplus a_{\varphi(i)1}} 
  \oplus \cdots \oplus G_{\varphi(i)r-1}^{\oplus a_{\varphi(i)r-1}},
  \end{align*}
  since the action of $F_{i0}$ sends every projective $Q_j$ to some number of copies of the 
  projective $Q_{\varphi(i)}$.

  Now, we have that $F_{00}$ is self-adjoint and thus we can obtain the matrix of $F_{00}$ in 
  the basis of simples by simply transposing $\mathsf{M}_0$, that is, it is given by 
  \begin{displaymath}
  \begin{pmatrix} 
    2 & 0 & \cdots & 0 \\
    2 & 0 & \cdots & 0 \\
    \vdots & \vdots & \ddots & \vdots \\
    2 & 0 & \cdots & 0
  \end{pmatrix}.
  \end{displaymath}

  Consequently, $F_{00}\hat{L}_s = 0$ if $s > 0$ and $F_{00}\hat{L}_0$ is isomorphic to either $Q_0$ or $Q_0 \oplus Q_0$.
  In the second case the socle of $Q_0$ must contain some $\hat{L}_s$ with $s > 0$ which is not possible as
  \begin{multline*}
  0 \neq \on{dim}(\operatorname{Hom}(\hat{L}_s,Q_0 \oplus Q_0)) = \on{dim}(\operatorname{Hom}(\hat{L}_s,F_{00}\hat{L}_0)) 
  = \\ = \on{dim}(\operatorname{Hom}(F_{00}\hat{L}_s,\hat{L}_0)) = \on{dim}(\operatorname{Hom}(0,\hat{L}_0)) = 0, 
  \end{multline*}
  Therefore $F_{00}\hat{L}_0 \simeq Q_0$ and hence the projective $Q_{0}$ has socle isomorphic to $\hat{L}_0$ and
  two copies of each simple $\hat{L}_s$ in its Jordan-H\"{o}lder series.
  In particular we see that 
  \begin{align*}
  2 = \on{dim}(\on{Hom}(Q_{0}, Q_{0})) = \on{dim}(\epsilon_{0}B\epsilon_{0}).
  \end{align*}
  From $F_{00}\hat{L}_0 \simeq Q_0$ and $F_{00}\hat{L}_s = 0$ if $s > 0$ we have $\overline{\mathbf{M}}(F_{00}) \simeq G_{00}$. 
  Further, from \eqref{eq:1} we have, for $i \neq 0$, 
  that $F_{00} \circ F_{i0} \simeq F_{00}$, which yields, by a similar argument as in the proof of 
  Lemma \ref{MatrixLemma2x2}, that $\overline{\mathbf{M}}(F_{i0}) \simeq G_{\varphi(i)0}$.
  This implies that 
  \begin{align*}
  \mathbf{M}(F_{i0})Q_{0} \simeq G_{\varphi(i)0}Q_{0} \simeq B\epsilon_{\varphi(i)} \otimes \epsilon_{0} B \otimes_B B\epsilon_{0} 
  \simeq B\epsilon_{\varphi(i)}^{\oplus \on{dim}(\epsilon_{0}B\epsilon_{0})} \simeq Q_{\varphi(i)}^{\oplus 2}.
  \end{align*}
  On the other hand, the form of $\mathsf{M}_{i}$ yields that $\mathbf{M}(F_{i0})Q_{\varphi(1)} = Q_{\varphi(i)}$, 
  a contradiction. This shows that the only possible matrices for $\mathbf{M}$ are the matrices given in \eqref{eq:MatrixCase1}. 
\end{proof}

\subsection{Speculations about a proof strategy for Conjecture \ref{GeneralCaseConjecture}}\label{s5.3}
Theorem \ref{MatrixThmGeneral} serves as evidence for Conjecture \ref{GeneralCaseConjecture} as the only matrices of 
the action of $F_{i0}$ are the ones which correspond to the cell $2$-representation for $\cC_k$, where $k \leq n$. This is 
an indication but not a proof that simple transitive $2$-representations of $\cC$ with these matrices exist. Furthermore, 
it is unclear whether, for each set of matrices, there only exists one simple transitive $2$-representation corresponding to this set.  

First, as already mentioned above, the matrices that are possible are exactly those of the simple transitive $2$-representations for the 
cases $k \leq n$. So it seems possible that there is a $2$-functor from $\cC_n$ to $\cC_k$ which maps different $F_{i0}$ of $\cC_n$ to
isomorphic $1$-morphisms in $\cC_k$. However, we do not know how to construct such a $2$-functor explicitly.

Now observe that, for each $n \geq 1$, we can associate to $\Lambda_n$ its $\Bbbk$-linear path category $\mathcal{A}_n$, i.e.,\ the category 
whose objects are the vertices of $\overline{S_n}$ and generating morphisms are the arrows of $\overline{S_n}$. 
Composition of morphisms is given as compositions of paths generated by the arrows our quiver and is subject to the relations 
descried in Section~\ref{s2.5}. 

Now observe that, given $n \geq k \geq 1$ and a surjection $\varphi: \{1, 2, \ldots, n\} \to \{2, 3, \ldots k\}$ of the kind
as described in the conjecture, there exists a functor $F_{\varphi}: \mathcal{A}_n \to \mathcal{A}_k$ which 
maps objects of $\mathcal{A}_n$ to objects of $\mathcal{A}_k$ according to $\varphi$ and it maps $a_i$ to $a_{\varphi(i)}$ and 
similarly for $b_i$. Now this functor induces a functor from $\mathcal{A}_k$-$\on{mod}$ to $\mathcal{A}_n$-$\on{mod}$, and the latter
categories are equivalent to $\Lambda_k$-$\on{mod}$ and $\Lambda_n$-$\on{mod}$, respectively. However, this functor does not map 
projectives to projectives and thus to construct the $2$-representation we want, some modifications are required. 

Similarly, there is a homomorphism of $\Bbbk$-algebras from $\Lambda_k$ to $\Lambda_n$ which maps, for $i \in \{2, 3, \ldots, k\}$,
the primitive idempotent corresponding to $i$ in $\Lambda_k$ to the sum of all primitive idempotents in the preimage of $i$ in $\Lambda_n$.
This homomorphism induces a functor from $\Lambda_n$-$\on{mod}$ to $\Lambda_k$-$\on{mod}$. 
Again this functor, unfortunately, does not map projectives to projectives and thus requires some adjustment in order to be used to 
construct a $2$-representation.

\section{The right cell case}\label{s6}
In this section we consider a different full $2$-subcategory of $\cC_{\Lambda_n}$. Here we consider the full $2$-subcategory
with respect to a right cell and not a left cell. More precisely we define
\[
 \cC^{r} := \cC_n^{r} = \text{add}(\mathbb{1}_{\mathtt{1}}, F_{00}, F_{01}, \ldots, F_{0n}).
\]
In this case we see that $\cC^{r}$ has the singleton cell containing only the identity which is a left, right and 
two-sided cell at the same time. Moreover, there is one more two-sided cell containing all other $1$-morphisms which also is a right 
cell and then each $F_{0j}$ gives forms its own singleton left cell $\l_j$. This can be seen easily when looking at the 
composition table of the $F_{0j}$'s:
\begin{align*}
F_{0i}\circ F_{0j} = b_iF_{0j}, \quad \text{where } b_i = \begin{cases}
                                                           2, \quad &\text{if } i = 0,\\
                                                           1, \quad &\text{else}.
                                                          \end{cases}
\end{align*}

Similarly to the above, we set $\displaystyle F := \bigoplus_{i = 0}^n F_{0i}$ and make the same assumptions about the algebra $B$.

\subsection{Matrix analysis}\label{s6.1}
Let $n \geq 1$. Then we know that $F^2 = F^{\oplus (n+2)}$ and thus, assuming $1$-faithfulness as above, 
we know that $\mathsf{M} = [F]$ is a positive integral matrix satisfying $\mathsf{M}^2 = (n+2)\mathsf{M}$.
Moreover, we have that 
\[
 \mathsf{M} = \mathsf{M}_0 + \mathsf{M}_1 + \cdots + \mathsf{M}_n,
\]
where $\mathsf{M}_i = [F_{0i}]$, i.e.,\ $\mathsf{M}$ is the sum of $n + 1$ non-negative matrices. By Lemma~\ref{lem1}, 
we know that $\mathsf{M}_0$ has all rows except the first one equal to zero and the first diagonal entry equals $2$. 
However, this implies that, for $i>0$, the matrix $2\mathsf{M}_i = \mathsf{M}_0 \cdot \mathsf{M}_i$ has only zero rows except 
for possibly the first row. Consequently, the only non-zero row of $\mathsf{M}$ is the first row. Hence $\mathsf{M}$ is a $1 \times 1$-matrix.
This yields the following.
\begin{proposition} \label{prop1right}
 Let $\mathbf{M}$ be a simple transitive $2$-representation of $\cC^r$. Then $\mathbf{M}$ has rank $1$.
\end{proposition}
Now, using this, we can prove the following.
\begin{theorem}
 Let $\mathbf{M}$ be a simple transitive $2$-representation of $\cC^r$. Then $\mathbf{M}$ is equivalent to a cell $2$-representation
 of $\cC^r$.
\end{theorem}
\begin{proof}
 Assume first that $\mathcal{J}_e = \on{add}(\mathbb{1}_{\mathtt{i}})$ is the apex of $\mathbf{M}$, i.e.,\ $\mathbf{M}(F_{0i}) = 0$ for 
 all $i$. Then the quotient of $\cC^r_n$ by the $2$-ideal generated by all $F_{0i}$ satisfies the assumptions of \cite[Theorem 18]{MM5}
 and hence $\mathbf{M}$ is equivalent to a cell $2$-representation.
 
 Thus we may now assume that $\mathbf{M}(F_{0i})$ is non-zero for all $i$. Then by Proposition~\ref{prop1right} we know that
 $\mathbf{M}$ has rank $1$. Moreover, we know that $\overline{\mathbf{M}}(F_{00})$ is self-adjoint and hence we have that $(2)$ 
 is also the matrix of $F_{00}$ in the basis of simples. Moreover, we have that $\overline{\mathbf{M}}(F_{00})$ sends the simple
 $\hat{L}_0$ to a projective, i.e. a number of copies of $Q_0$. As $\overline{\mathbf{M}}(F_{00})\hat{L}_0$ contains two copies of 
 $\hat{L}_0$ and nothing else in its Jordan-H\"{o}lder series, we have that either $\hat{L}_1 = Q_0$ or that $Q_0$ has Loewy length 
 two. In the first case we have that $B$ is isomorphic to $\Bbbk$ and in the second case to $\Bbbk[x]\slash(x^2)$. We can now exclude 
 the first case as there is no simple transitive $2$-representation of $\on{add}(\mathbb{1}_{\mathtt{i}}, F_{00})$ on $\Bbbk$-proj. Indeed,
 the $2$-category $\on{add}(\mathbb{1}_{\mathtt{i}}, F_{00})$ satisfies the conditions of \cite[Theorem~18]{MM5} and hence all of its simple transitive 
 $2$-representations are equivalent to cell $2$-representations, one of which annihilates $F_{00}$ and the other one is
 on $\Bbbk[x]\slash(x^2)$-$\on{mod}$. This shows that $\overline{\mathbf{M}}(F_{00})\hat{L}_0 = Q_0$. Now, we can 
 define 
 \begin{align*}
  \Phi: \mathbf{P}_{\mathtt{i}} & \to \overline{\mathbf{M}},\\
  \mathbb{1}_{\mathtt{i}} & \mapsto \hat{L}_0.
 \end{align*}
 We have just seen that this restricts to a bijective map from $\on{add}(\mathbb{1}_{\mathtt{i}}, F_{00})$ to 
 $\overline{\mathbf{M}}_\text{proj}$ as the Cartan matrices in both cases are $(2)$, we have that $\Phi$ induces an equivalence 
 between $\cC_{\l}$ and $\overline{\mathbf{M}} \simeq \mathbf{M}$.
\end{proof}

{\bf Acknowledgments.}  The author wants to thank his supervisor Volodymyr Mazorchuk for many helpful discussions 
and comments on early versions of this manuscript. 

\bibliographystyle{alpha}

\end{document}